\documentclass[10pt]{amsart} 
\usepackage[margin=1in,letterpaper,portrait]{geometry}

\usepackage{latexsym,mathtools}
\usepackage{epsfig,setspace}
\usepackage{a4}
\usepackage{amssymb}
\usepackage{amsthm}
\usepackage{amsbsy} 
\usepackage{upgreek}
\usepackage{relsize}
\usepackage{tikz-cd}
\usepackage{ytableau}
\usepackage{amsfonts,amssymb,latexsym,amscd,amsmath,euscript,enumerate,verbatim,calc}
\allowdisplaybreaks
\usepackage{url}
\usepackage{hhline}
\usepackage{pifont}
\usepackage[toc,page]{appendix} 
\usepackage[colorlinks=true,linkcolor=blue,citecolor=magenta]{hyperref}

\theoremstyle{definition}

\newtheorem*{theorem*}{Theorem} 
\newtheorem{theorem}{Theorem}[section]    
\newtheorem{lemma}[theorem]{Lemma}
\newtheorem{proposition}[theorem]{Proposition}

\newtheorem{corollary}[theorem]{Corollary}
\newtheorem{defn}[theorem]{Definition}

\newtheorem{remark}[theorem]{Remark}

\def\Fq{{\mathbb F}_q}

\def\I{\mathcal{I}}

\newcommand{\diag}{\operatorname{diag}}

\newcommand{\GL}{\operatorname{GL}}
\newcommand{\T}{\operatorname{T}}
\newcommand{\M}{\operatorname{M}}

\allowdisplaybreaks
\makeatletter
\def\imod#1{\allowbreak\mkern10mu({\operator@font mod}\,\,#1)} 
\makeatother

\title[Polynomial Matrices and Splitting Subspaces]{Polynomial Matrices, Splitting Subspaces and Krylov Subspaces over Finite Fields} 

\author{Divya Aggarwal} 
\address{Indraprastha Institute of Information Technology Delhi (IIIT-Delhi), New Delhi 110020, India.}
\email{divyaa@iiitd.ac.in} 

\author{Samrith Ram}
\address{Indraprastha Institute of Information Technology Delhi (IIIT-Delhi), New Delhi 110020, India.}
\email{samrith@gmail.com}



\begin{document}
\begin{abstract}
  Let $T$ be a linear operator on an $\Fq$-vector space $V$ of dimension $n$. For any divisor $m$ of $n$, an $m$-dimensional subspace $W$ of $V$ is $T$-\emph{splitting} if
  $$
V =W\oplus TW\oplus \cdots \oplus T^{d-1}W,
$$
where $d=n/m$. Let $\sigma(m,d;\T)$ denote the number of $m$-dimensional $T$-splitting subspaces. Determining $\sigma(m,d;T)$ for an arbitrary operator $T$ is an open problem. This problem is closely related to another open problem on Krylov spaces. We discuss this connection and give explicit formulae for $\sigma(m,d;T)$ in the case where the invariant factors of $T$ satisfy certain degree conditions. A connection with another enumeration problem on polynomial matrices is also discussed.
\end{abstract}
\maketitle
\tableofcontents 

\section{Introduction}

Denote by $\Fq$ the finite field with $q$ elements and by $\Fq[x]$ the $\Fq$-algebra of polynomials in the indeterminate $x$. In what follows, $n,k,m,d$ will denote positive integers unless otherwise stated. For any ring $R$, the set of all $n\times k$ matrices over $R$ is denoted by $M_{n,k}(R)$ while $M_n(R)$ indicates the ring of $n\times n$ matrices over $R$. In what follows, we will always assume $k\leq n$. Define
\begin{align*}
 \M_q(n,k,d):=\{x^dI+x^{d-1}C_{d-1}+\cdots+C_0:C_i\in\M_{n,k}(\Fq)\},
\end{align*}
where $I$ denotes the $n\times k$ matrix whose $(i,j)$-th entry is 1 if $i=j$ and 0 otherwise. Evidently any element of $\M_q(n,k,d)$ is a polynomial with matrix coefficients which may also be viewed as a single $n\times k$ matrix over $\Fq$ whose entries are polynomials in $x$. Given elements $P,Q\in \M_q(n,k,d)$ write $P\sim Q$ ($P$ is equivalent to $Q$) if there exist invertible matrices $A\in M_n(\Fq[x])$ and $B\in M_k(\Fq[x])$ such that $APB=Q$. It can be shown that each $P\in M_q(n,k,d)$ is equivalent to a diagonal matrix
$$
P\sim \diag_{n,k}(p_1,\ldots,p_k),
$$
where $p_1,\ldots,p_k$ are monic polynomials over $\Fq$ satisfying $p_i\mid p_{i+1}$ for $1\leq i <k$. This diagonal form is called the Smith Normal Form \cite[p. 260]{MR0276251} of $P$. By a $k$-tuple of invariant factors, we mean a $k$-tuple $\I=(f_1,\ldots,f_k)$ where each $f_i$ is a monic polynomial over $\Fq$ and $f_i\mid f_{i+1}$ for $1 \leq i \leq k-1$. Given a $k$-tuple $\I=(p_1,\ldots,p_k)$ of invariant factors, define
\begin{equation*}
  \label{eq:nkdI}
 \mu_q(n,k,d;\I):=|\left\{P \in\M_q(n,k,d):P\sim\diag_{n,k}(p_1,\ldots,p_k)\right\}|.
\end{equation*} 
In other words, $\mu_q(n,k,d;\I)$ is the number of elements in $\M_q(n,k,d)$ whose Smith form comprises precisely the polynomials appearing in $\I$ as diagonal entries. The quantity $\mu_q(n,k,d;\I)$ is the main object of study in this paper. Determination of $\mu_q(n,k,d;\I)$ given an arbitrary assignment of the parameters in full generality is an open problem. To underscore the significance of studying $\mu_q$ we briefly mention specific cases in the literature where it has been considered previously in the context of group theory, probability theory, unimodularity and mathematical control theory. 

For $d=1$ and $k=n$, we have
$$
 \mu_q(n,n,1;\I):=|\left\{C_0 \in M_n(\Fq):xI+C_0\sim\diag_{n}(p_1,\ldots,p_k)\right\}|.
$$
Two matrices $A,B\in M_n(\Fq)$ are similar if and only if $xI-A$ and $xI-B$ have the same Smith form. Thus in this setting we have the matrix conjugacy class size problem: How many matrices are similar to a given matrix $A$ with invariant factors $\I=(p_1,\ldots,p_n)$? Denoting by $c(\I)$ the size of the centralizer in the general linear group $\GL_n(\Fq)$ corresponding to the conjugacy class indexed by $\I$, we have
$$
\mu_q(n,n,1;\I)=\frac{|\GL_n(\Fq)|}{c(\I)}.
$$
According to Stanley \cite[p. 108]{MR2868112} a precise expression (see \eqref{c(I)}) for the size of the centralizer $c(\I)$ was first given by Philip Hall based on earlier work by Frobenius.

The case $k=1$ is considered in Section \ref{centralizer}. In this case, we must have $\I=(g)$ for some monic polynomial $g$ and in this case it is not difficult to see that $\mu_q(n,1,d,(g))$ counts the number of $n$-tuples $(g_1,\ldots,g_n)$ of monic polynomials over $\Fq$ of degree $d$ such that $\gcd(g_1,\ldots,g_n)=g$. In particular, for $g=1$, this problem has been studied by Corteel, Savage, Wilf and Zeilberger \cite[Prop. 3]{MR1620873} and a nice answer is known in this case.

The case where $\I$ is a $k$-tuple of 1's corresponds to unimodularity. A polynomial matrix is \emph{unimodular} if its maximal minors are coprime. The case $d=1$ for arbitrary $n,k$ with $k<n$ and $\I=(1,\ldots,1)$ has been considered, albeit in a slightly different context, by Lieb, Jordan and Helmke \cite[Thm. 1]{MR3471061} who essentially prove that if $k<n$ are positive integers, then the number of matrices $A\in M_{n,k}(\Fq)$ for which $xI-A$ is unimodular is given by
  $$
\mu_q(n,k,1,(1,\ldots,1))=\prod_{i=1}^k (q^n-q^i).
$$
 This theorem has connections with mathematical control theory and answers a question of Kociecki and Przyłuski \cite{MR1019984} on the number of reachable linear systems over a finite field. We refer to the introduction of \cite{MR3462983} for these connections and the link with unimodularity. A recent generalization of this result in the setting of unimodular polynomial matrices appears in \cite[Thm. 4.1]{arora2020unimodular} and corresponds to the case of general $d$: 
  $$
\mu_q(n,k,d,(1,\ldots,1))=q^{nk(d-1)}\prod_{i=1}^k (q^n-q^i).
$$

It follows from Theorem \ref{thm1} that determining $\mu_q(n,k,d;\I)$ for $n=k$ is intimately connected with an open problem on splitting subspaces. To state the problem we require the definition of a splitting subspace.
 \begin{defn}
Let $T$ be a linear operator on an $md$-dimensional vector space $V$ over $\Fq$. An $m$-dimensional subspace $W$ of $V$ is said to be $T$-splitting if
$$
V =W\oplus TW\oplus \cdots \oplus T^{d-1}W.
$$
 \end{defn}  
 Let $\sigma(m,d;T)$ denote the number of $m$-dimensional $T$-splitting subspaces. Determining $\sigma(m,d;T)$ for an arbitrary operator $T$ is an open problem \cite[p. 54]{split} . The case $d=1$ is trivial while the case $m=1$ is considered in \cite[Prop. 4.4]{split}. The case where $T$ has an irreducible characteristic polynomial was settled by Chen and Tseng \cite[Cor. 3.4]{MR3093853} who proved a conjecture made in \cite{MR2831705}; the case of cyclic nilpotent $T$ has recently been settled in~\cite[Cor. 4.7]{aggarwal2021splitting}. Splitting subspaces were originally defined in slightly less generality by Niederreiter~\cite[p. 54]{MR1334623} in connection with his work on pseudorandom number generation. Splitting subspaces are also closely related to anti-invariant subspaces (Definition \ref{antiinvariant}) studied by Barrià and Halmos \cite{MR748946}, Sourour \cite{MR822138} and Knüppel and Nielsen~\cite{MR2013452}. The problem of determining $\sigma(m,d;T)$ has connections with an important unsolved problem on Krylov subspaces which we now discuss.        

  Let $T$ be a linear operator on an $N$-dimensional vector space $V$ over $\Fq$. Let $S=\{v_1,\ldots,v_m \}$ be a set of $m$ vectors in $V$. The {\it truncated Krylov subspace} \cite[p. 277]{MR1982139} of order $d$ generated by $S$ is defined by
 $$
 \mathrm{Kry}(T,S;d):=\left\{\sum_{i=1}^{m}f_i(T)v_i:f_i(x)\in \Fq[x]\mbox{ and } \deg f_i<d\right\}.
  $$
Define
  $$
\kappa_{m,d}(T):=\frac{1}{q^{Nm}}|\{(v_1,\ldots,v_m)\in V^m:\mathrm{Kry}(T,\{v_1,\ldots,v_m\};d)=V\}|.
$$
The number $\kappa_{m,d}(T)$ may be interpreted as the probability of selecting $m$ vectors $v_1,\ldots,v_m$ uniformly and independently from $V$ such that the truncated Krylov subspace of order $d$ spanned by them is all of $V$. Determining $\kappa_{m,d}(T)$ is useful in solving large sparse linear systems over finite fields which arise frequently in number theory and computer algebra. Krylov-based methods such as Wiedemann's algorithm are used to compute the minimal polynomials of large matrices over finite fields \cite{MR3424032}. For another instance, the Number Field Sieve which is a classical algorithm for factoring large integers relies on Krylov subspace methods \cite[p. 24]{bouillaguet2020parallel}.
 As the probability $\kappa_{m,d}(T)$ is relevant to the analysis of the efficiency of such algorithms, obtaining bounds on $\kappa_{m,d}(T)$ is a difficult and important \cite[p. 277]{MR1982139} problem. 
Proposition \ref{prop:krylovsigma} connects $\kappa_{m,d}(T)$ to splitting subspaces:
\begin{equation}
  \label{kappasigma}
 \kappa_{m,d}(T)=\frac{|\GL_m(\Fq)| \cdot \sigma(m,d;T)}{q^{m^2 d}}.  
\end{equation}
 In this paper, we prove that if $T$ is a linear operator on an $md$ dimensional vector space $V$ over $\Fq$, then $\sigma(m,d;T)>0$ if and only if the number of nonconstant invariant factors of $T$ is at most $m$. It is easily seen that the number $\sigma(m,d;T)$ depends only on the similarity class of $T$ \cite[Prop. 3.2]{aggarwal2021splitting} since a subspace $W$ is $T$-splitting if and only if $SW$ is $S\circ T\circ S^{-1}$ splitting for each linear isomorphism $S$ of $V$. Thus, given an $md$-tuple of invariant factors $\I$, one can define
 $$
\sigma(m,d;\I)=\sigma(m,d;T),
 $$
 where $T$ is any linear operator with invariant factors $\I$. By the positivity criterion for $\sigma(m,d;T)$ above, we may restrict ourselves to the case where the first $m(d-1)$ coordinates of the $md$-tuple $\I$ are equal to 1:
 \begin{align*}
 \I=(1,1,\ldots,1,p_1,p_2,\ldots,p_m).
 \end{align*}
For $\I$ as above, we prove the following.
\begin{theorem}
  If $\deg p_1=d$ and $p_1=g_1^{e_1}\cdots g_t^{e_t},$ where the $g_i$ are distinct irreducible polynomials with $\deg g_i=d_i(1\leq i\leq t)$, then
  $$
\sigma(m,d;\I)= \frac{\prod_{i=1}^{t}\prod_{j=1}^{m}(1-q^{-jd_i})}{\prod_{j=1}^{m}(1-q^{-j})}q^{m^2(d-1)}.
  $$
\end{theorem}
\begin{theorem}
  If $\deg p_1=d-1$, then
  $$
\sigma(m,d;\I)=\frac{c(\I)}{c(\widetilde{\I})},
$$
where $\widetilde{\I}=(\tilde{p}_1,\ldots,\tilde{p}_m)$ with $\tilde{p}_i=p_i/p_1$ for $1\leq i\leq m$.
\end{theorem}
Note that $c(\I)$ (see \eqref{c(I)} for a precise expression) corresponds to a centralizer in $\GL_{md}(\Fq)$ while $c(\widetilde{\I})$ is associated with the group $\GL_m(\Fq)$. In conjunction with~\eqref{kappasigma}, the theorems above can be used to derive precise formulae for the probability $\kappa_{m,d}(T)$ for suitable values of $T$. 

\section{Existence of Splitting Subspaces}
\label{pre}

  If $T$ is a linear operator on a vector space $V$, then $T$ defines an $\Fq[x]$-module on the vector space $V$, where the action of $x$ is defined by $x\cdot v= Tv$ for $v \in V$. The $\Fq[x]$ module $V$ can be decomposed as a direct sum

$$V\simeq \bigoplus_{i=1}^{t}\bigoplus_{j=1}^{\ell_i} \frac{\Fq[x]}{(\phi_i^{\lambda_{i,j}})},$$
where $\phi_1, \ldots, \phi_t$ are distinct monic irreducible polynomials and, for each $i$, $\lambda_i= (\lambda_{i,1}, \lambda_{i,2}, \ldots ,\lambda_{i,\ell_i})$ is an integer partition corresponding to $\phi_i(1\leq i\leq t)$.

The finite set $\{(\phi_1,\lambda_1),\ldots,(\phi_t,\lambda_t)\}$ completely determines the {\it similarity class} of $T$ and corresponds uniquely to the invariant factors of $T$, i.e., the invariant factors of $xI-A$ where $A$ is the matrix of $T$ with respect to some basis. If $d_i=\deg \phi_i$, then the finite multiset 
$$
\tau=\{(d_1,\lambda_1),\ldots,(d_t,\lambda_t)\}
$$ 
is called the {\it similarity class type} (or simply type) of the linear operator $T$. The notion of similarity class type may be traced back to the work of Green \cite[p. 405]{MR72878} on the irreducible characters of the finite general linear groups. One of the main reasons for considering the similarity class type is that many combinatorial invariants associated with $T$ often depend only on the partitions $\lambda_i$ and the degrees of the polynomials $\phi_i$ (and not the polynomials themselves). The {\it size} of a similarity class type is the dimension of the vector space on which the corresponding operator is defined.
For any $k$-tuple of invariant factors $\I=(p_1,\ldots,p_k)$, define
$$
\deg \I := \deg (p_1 \cdots p_k).
$$
As stated in the introduction, the number of splitting subspaces $\sigma(m,d;T)$ depends only on the similarity class of $T$. Here is a precise definition.
\begin{defn}
\label{sigmaI}
Let $\I$ be an $md$-tuple of invariant factors with $\deg \I=md$. Then
$$
\sigma(m,d;\I) := \sigma(m,d;T),
$$
where $T$ is any linear operator on an $md$-dimensional vector space over $\Fq$ with invariant factors $\I$. 
\end{defn}
More generally, it is known \cite[Cor. 3.7]{aggarwal2021splitting} that the number of splitting subspaces $\sigma(m,d;T)$ depends only on the similarity class type of $T$. In other words, whenever there exists a linear transformation over $\Fq$ of similarity class type $\tau$, one can define
$$
\sigma(m,d;\tau):=\sigma(m,d;T),
$$
where $T$ is any linear operator of type $\tau$ defined on an $md$-dimensional vector space over $\Fq$.
For our purposes it will be more convenient to work with $\sigma(m,d;\I)$ but it is worth emphasizing that all results involving $\sigma(m,d;\I)$ may be reformulated in terms of similarity class type and we will include such examples.

Splitting subspaces are closely related to anti-invariant subspaces. 
\begin{defn}
  \label{antiinvariant}
Given a non-negative integer $\ell$, a subspace $W$ of $V$ is called \emph{$\ell$-fold $T$-anti-invariant} if 
$$
\dim(W+TW+\cdots+T^{\ell}W)=(\ell+1)\cdot \dim W.
$$
\end{defn} 
If $\dim V = md$, then every $m$-dimensional $T$-splitting subspace is $(d-1)$-fold $T$-anti-invariant.
Barría and Halmos \cite{MR748946} and Sourour \cite{MR822138} studied 1-fold anti-invariant subspaces and determined the maximum possible dimension of such a subspace. Knüppel and Nielsen (2003) extended their work to $\ell$-fold anti-invariant subspaces for arbitrary $\ell$.
In particular, they gave the following existence criterion.

\begin{proposition}\cite[Cor. 2.2]{MR2013452}
\label{existence}
Let $T$ be a linear operator on an $N$-dimensional vector space over $\Fq$ where $N=(\ell+1)m$. Suppose $(p_1,\ldots,p_N)$ is the $N$-tuple of invariant factors of $T$. Then an $\ell$-fold $T$-anti-invariant subspace of dimension $m$ exists if and only if $p_i=1$ for $1\leq i\leq \ell m$.
\end{proposition}

It is important to note that \cite[Cor. 2.2]{MR2013452} is stated only for invertible operators. However, the proof \cite[Lem 3.1, Lem 3.2]{MR2013452} does not require the hypothesis that $T$ is invertible. 

Proposition \ref{existence} yields a criterion for the existence of splitting subspaces. 
\begin{corollary}
\label{im}
Let $\I=(p_1,\ldots,p_{md})$ with $\deg \I=md$. Then $\sigma(m,d;\I)>0$ if and only if $p_i=1$ for $1\leq i\leq m(d-1)$.
\end{corollary}
 
\begin{corollary}
  If $\tau=\{(d_1,\lambda_1),\ldots,(d_t,\lambda_t)\}$, then $\sigma(m,d;\tau)>0$ if and only if each partition $\lambda_i(1\leq i\leq t)$ has at most $m$ parts.
\end{corollary}

Recall the probability 
  $$
\kappa_{m,d}(T):=\frac{1}{q^{Nm}}|\{(v_1,\ldots,v_m)\in V^m:\mathrm{Kry}(T,\{v_1,\ldots,v_m\};d)=V\}|
$$
 defined in the introduction. The next proposition gives the connection between $\kappa_{m,d}(T)$ and $\sigma(m,d;T)$. In what follows we will denote the order of the general linear group $\GL_m(\Fq)$ by $\gamma_q(m) = (q^m - 1) (q^m - q) \ldots (q^m - q^{m-1})$.

\begin{proposition}
\label{prop:krylovsigma}  
Let $T$ be a linear operator on an $md$-dimensional vector space over $\Fq$. Then
$$
\kappa_{m,d}(T)=\frac{\gamma_q(m) \cdot \sigma(m,d;T)}{q^{m^2 d}}.
$$
\end{proposition}

\begin{proof}
Let $S=\{v_1,\ldots,v_m\}\subseteq V$. By the definition of truncated Krylov subspace, we have $\mathrm{Kry}(T,S;d)=V$ if and only if $W=\mathrm{span}(S)$ is an $m$ dimensional $T$-splitting subspace. Since each $m$-dimensional subspace has precisely $\gamma_q(m)$ ordered bases, the proposition follows.
\end{proof}

\section{Splitting Subspaces and Polynomial Matrices}
\label{main}

In this section we show that $\sigma(m,d;\I)$ may be recovered from $\mu_q(m,m,d;\I)$, where $\mu_q$ is as defined in the introduction. We begin with the following lemma concerning the equivalence of matrices.

\begin{lemma}\label{lem:bcm}\cite[Thm. 1.1]{MR3396732}
Let $P = x^d I+x^{d-1}C_{d-1}+\cdots+C_0 \in\M_q(m,m,d)$. Consider the $md \times md$ block matrix
$$
A=
\begin{bmatrix}
\bf{0} & \bf{0} & \dots & \bf{0} & -C_0 \\
{I}& \bf{0} & \dots & \bf{0} & -C_1 \\
\bf{0} & {I} & \dots & \bf{0} & -C_2 \\
\vdots & \vdots & \ddots & \vdots & \vdots \\
\bf{0} & \bf{0} & \dots & {I} & -C_{d-1}
\end{bmatrix}.
$$
Then 
$$
xI - A \sim 
\begin{bmatrix}
I & \bf{0}\\
\bf{0} & P
\end{bmatrix},
$$
where $I$ and $\bf{0}$ denote the identity and zero matrices of appropriate sizes. In particular, if $P$ has invariant factors $(p_1,\ldots,p_m)$, then $xI-A$ has invariant factors $(1,\ldots,1,p_1,\ldots,p_m).$
\end{lemma}
Given $\I=(p_1,\ldots,p_m)$, the above lemma implies that $\mu_q(m,m,d;\I)$ is equal to the number of matrices $A$ of the block form above whose invariant factors are $\I=(1,\ldots,1,p_1,\ldots,p_m)$. In view of Corollary \ref{im}, given any $m$-tuple $\I=(p_1,\ldots,p_m)$ we interpret $\sigma(m,d;\I)$ to mean $\sigma(m,d;\I')$ where $\I'=(1,\ldots,1,p_1,\ldots,p_m)$ is the $md$-tuple obtained from $\I$ by padding $m(d-1)$ ones. As a natural extension we will write $c(\I)$ to mean $c(\I')$ hereafter.
\begin{theorem}
\label{thm1}
Let $\I=(p_1,\ldots,p_m)$ be an $m$-tuple of invariant factors with $\deg \I=md$. Then
$$
\sigma(m,d;\I)=\frac{c(\I)}{\gamma_q(m)}\mu_q(m,m,d;\I).
$$
\end{theorem}

\begin{proof}
Let $T$ be a linear operator on an $md$-dimensional vector space with invariant factors $\I$ and let $W$ be an $m$-dimensional $T$-splitting subspace. Suppose $\mathcal{B}_W=\{v_1,\ldots,v_m\}$ is an ordered basis for $W$. Then an ordered basis for $V$ is given by
$$
\mathcal{B}_V= \{ v_1,\ldots,v_m, Tv_1,\ldots,Tv_m,\ldots,T^{d-1}v_1,\ldots,T^{d-1}v_m \}.
$$ 
The matrix of $T$ with respect to the basis $\mathcal{B}_V$ has the block form
\begin{equation}
  \label{eq:bcm}
\begin{bmatrix}
\bf{0} & \bf{0} & \dots & \bf{0} & -C_0 \\
I & \bf{0} & \dots & \bf{0} & -C_1 \\
\bf{0} & I & \dots & \bf{0} & -C_2 \\
\vdots & \vdots & \ddots & \vdots & \vdots \\
\bf{0} & \bf{0} & \dots & I & -C_{d-1}
\end{bmatrix},
\end{equation}
for some matrices $C_0, C_1, \ldots, C_{d-1} \in \M_m(\Fq)$.
Conversely, if $\{ \alpha_1, \ldots, \alpha_{md} \}$ is an ordered basis for $V$ with respect to which the matrix of $T$ is in the above block form, then span $\{ \alpha_1, \ldots, \alpha_{m} \}$ forms a $T$-splitting subspace for $V$.
Since there are $\sigma(m,d;\I)$ splitting subspaces of dimension $m$ and each such subspace has $\gamma_q(m)$ bases, it follows that $V$ has $\sigma(m,d;\I) \cdot \gamma_q(m)$ bases with respect to which the matrix of $T$ has the above block form. Different bases for $V$ may yield the same matrix for $T$. If $A$ denotes the matrix of $T$ with respect to the basis $\mathcal{B}_V$, then the number of different bases $\mathcal{B}$ for $V$ such that the matrix of $T$ with respect to $\mathcal{B}$ is $A$ is precisely
\begin{align*}
|\{ P \in \GL_{md}(\Fq) : P^{-1}AP=A \}|=c(\I).
\end{align*}
By Lemma \ref{lem:bcm}, $\mu_q(m,m,d;\I)$ is equal to the number of matrices of the block form~\eqref{eq:bcm} whose invariant factors are $\I$. Thus
$$
\mu_q(m,m,d;\I)=\frac{\sigma(m,d;\I) \cdot \gamma_q(m)}{c(\I)},
$$
which proves the theorem.
\end{proof}

\begin{defn}
Let $P \in\M_q(n,k,d)$ and suppose $1 \leq i \leq k$. The $i^{\text{th}}$ determinantal divisor of $P$, denoted $\delta_i(P)$, is the greatest common divisor of all $i \times i$ minors of $P$. 
\end{defn}
The invariant factors of $P$ can be recovered from the determinantal divisors; if $P \in \M_q(n,k,d;(p_1,\ldots,p_k))$, then \cite[p. 260]{MR0276251}
$$
p_i = \frac{\delta_i(P)}{\delta_{i-1}(P)} \quad (1 \leq i \leq k),
$$
where $\delta_0(P)=1$. In what follows, $\M_q(n,k,d;\I)$ denotes all elements $P\in \M_q(n,k,d)$ which have invariant factors $\I$. By definition, we have $\mu_q(n,k,d;\I)=|M_q(n,k,d;\I)|$.
\begin{remark}
  \label{rem:maxdeg}
If $P\in \M_q(n,k,d;\I)$, then
$$
\deg \I=\deg\prod_{i=1}^k p_i =\deg\delta_k(P)\leq kd,
$$
since the maximum possible degree of a minor of $P$ is precisely $kd$. In view of this degree constraint above we will implicitly assume that $\deg\I\leq kd$ whenever $\mu_q(n,k,d;\I)$ is considered.
\end{remark}
The following reduction lemma will prove very useful.
\begin{lemma}
\label{star}
We have
$$
\mu_q(n,k,d;(p_1,\ldots,p_k))=\mu_q(n,k,d-d_1;(\tilde{p}_1,\ldots,\tilde{p}_k)),
$$
where $d_1=\deg p_1$ and $\tilde{p}_i=p_i/p_1$ for $1 \leq i \leq k$.
\end{lemma}

\begin{proof}
Let $\I=(p_1,\ldots,p_k)$ and suppose $P \in\M_q(n,k,d;\I)$. Since $\delta_1(P)=p_1$, it follows that $P=p_1 \cdot Q$ for some  $Q \in \M_q(n,k,d-d_1)$. It is easily seen that $\delta_i(Q)=\delta_i(P)/p_1^i$ for $1 \leq i \leq k$. Therefore the invariant factors for $Q$ are
$$
\left(\frac{p_1}{p_1},\frac{p_2}{p_1},\ldots,\frac{p_k}{p_1}\right)=\left(\tilde{p}_1,\ldots,\tilde{p}_k \right)=\widetilde{\I}.
$$
It follows that the map $\M_q(n,k,d;\I)\to \M_q(n,k,d-d_1;\widetilde{\I})$ defined by $P\mapsto P/p_1$ is a bijection.
\end{proof}

\begin{corollary}
\label{cor:ggg}  
Let $g$ be a monic polynomial of degree $d$ over $\Fq$ and  $\I=(g,\ldots,g)$ be a $k$-tuple of invariant factors. Then
$$
\mu_q(n,k,d;\I)=1.
$$
\end{corollary}

\begin{proof}
By Lemma \ref{star}, $\mu_q(n,k,d;(g,\ldots,g))=\mu_q(n,k,0;(1,\ldots,1))$ and the corollary follows since $\mu_q(n,k,0;(1,\ldots,1))=1$ by definition.
\end{proof}

A precise expression for the cardinality of the centralizer in $\GL_n(\Fq)$ of a matrix $A \in \M_n(\Fq)$ is known. As we will require this expression in some calculations, we state it here. Suppose the similarity class of $A$ is given by $\{ (\phi_1,\lambda_1),\ldots,(\phi_t,\lambda_t)\}$ for some irreducible polynomials $\phi_i$ and partitions $\lambda_i(1\leq i\leq t)$. Let the corresponding invariant factors (i.e. those of $xI-A$) be $\I=(p_1,\ldots,p_{n})$. For any partition $\lambda$, let $\lambda'$ denote its conjugate partition and let $m_i(\lambda)=\lambda'_i-\lambda'_{i+1}$ denote the number of parts of $\lambda$ of size $i$. Denote by $\langle \lambda, \lambda \rangle$ the sum of squares of the parts of $\lambda$.
For an indeterminate $u$ and a non-negative integer $r$, define
\begin{align*}
\left(u\right)_r :=\prod_{i=1}^{r}(1-u^i).
\end{align*}  
For any monic irreducible polynomial $\phi$ of degree $d$ and integer partition $\lambda$, let
$$
c_d(\lambda) = q^{d \langle \lambda', \lambda' \rangle}\prod_{i \geq 1} \left({q^{-d}}\right)_{m_i(\lambda)}.
$$
If $d_i=\deg \phi_i$, then the order of the centralizer of $A$ is given by \cite[p. 55]{Fulman2002}
\begin{align}
\label{c(I)}
c(\I)=\prod_{i=1}^t c_{d_i}(\lambda_i).
\end{align}

\begin{remark}
  \label{rem:ctau}
Note that the expression for $c(\I)$ above involves only the degrees of the polynomials $\phi_i$ and not the polynomials themselves. Thus given any similarity class type $\tau$, the corresponding centralizer size $c(\tau)$ is given by the product in \eqref{c(I)}. 
\end{remark}

 Corollary \ref{cor:ggg} and Theorem \ref{thm1} imply a formula for $\sigma(m,d;(p_1,\ldots,p_m))$ in the case where $\deg p_1=d$. In this case we necessarily have $p_i=p_1$ for each $1\leq i\leq m$.
\begin{theorem}
  \label{lastmsame}
  Let $g$ be a monic polynomial of degree $d$ over $\Fq$, and suppose $g=\phi_1^{e_1}\cdots \phi_t^{e_t}$ for distinct irreducible polynomials $\phi_i$ with $\deg \phi_i=d_i$ $(1 \leq i \leq t)$. If $\I$ denotes the $m$-tuple $(g,\ldots,g)$, then
  $$
\sigma(m,d;\I)=\frac{q^{m^2 d}}{\gamma_q(m)} \prod_{i=1}^{t} \prod_{j=1}^m\left(1-q^{-jd_i}\right).
  $$
\end{theorem}  
\begin{proof}
 By Corollary \ref{cor:ggg}, we have $\mu_q(m,m,d,(g,\ldots,g))=1$. The conjugacy class data corresponding to $\I$ is $\{(\phi_i,\lambda_i)\}_{1\leq i\leq t}$ where $\lambda_i$ is the integer partition with $m$ equal parts $e_i$ for each $1\leq i\leq t$. If $d_i=\deg \phi_i$, then by Theorem \ref{thm1} and \eqref{c(I)},
we have
\begin{align*}
  \sigma(m,d;\I)=\frac{c(\I)}{\gamma_q(m)} &=\frac{\prod_{i=1}^{t}c_{d_i}(\lambda_i)}{\gamma_q(m)}\\
  &=\frac{\prod_{i=1}^{t}q^{d_i m^2 e_i} \left({q^{-d_i}}\right)_m}{\gamma_q(m)}\\
  &=\frac{q^{m^2 \sum  d_ie_i}}{\gamma_q(m)}\prod_{i=1}^{t}\left({q^{-d_i}}\right)_m,
\end{align*}
and the theorem follows since $\sum d_ie_i=\deg g=d.$
\end{proof}

We may recast Theorem \ref{lastmsame} in terms of similarity class type.
\begin{theorem}
Let $\tau=\{(d_1,\lambda_1),\ldots,(d_t,\lambda_t)\}$ be a similarity class type of size $md$. If $\sum_{i=1}^t d_i \lambda_{i,m} =d$, then
$$
\sigma(m,d;\tau)=\frac{q^{m^2 d}}{\gamma_q(m)} \prod_{i=1}^{t} \prod_{j=1}^m\left(1-q^{-jd_i}\right).
$$
\end{theorem}

\section{Splitting Subspaces and Centralizers} 
\label{centralizer}
In this section we extend the definition of $c(\I)$ to include $k$-tuples of invariant factors $\I$ for which $\deg \I$ is not necessarily equal to $k$ in a natural way. If $\I=(p_1,\ldots,p_k)$ and $\deg \I=\delta>k$, then we set $c(\I)=c(\I')$ where $\I'=(1,\ldots,1,p_1,\ldots,p_k)$ denotes the $\delta$-tuple obtained by padding $\delta-k$ ones to $\I$. On the other hand, if $\delta < k$, then we set $c(\I)=c(\I')$ where $\I'$ is the $\delta$-tuple $(p_{k-\delta+1},\ldots,p_k)$. Denote by $\mathbb{I}_q(n,k,d)$ the set of all possible $k$-tuples of invariant factors that arise as the invariant factors of some element $P\in \M_q(n,k,d)$.

For $\I=(p_1,\ldots,p_k)$, the definition of $\mu_q(n,k,d;\I)$ states that
$$
\mu_q(n,k,1;\I)=|\left\{A \in \M_{n,k}(\Fq):xI-A\sim\diag_{n,k}(p_1,\ldots,p_k)\right\}|.
$$
A precise formula for $\mu_q(n,k,1;\I)$ was originally given in \cite[Thm. 3.8]{RAM2017146}. 

\begin{theorem}
\label{thmLAA}
Let $\I\in \mathbb{I}_q(n,k,1)$ and suppose $\deg \I=\delta$. Then
$$
\mu_q(n,k,1;\I)={k \brack \delta}_q \frac{\gamma_q(\delta)}{c(\I)} \prod_{i=\delta+1}^{k}(q^n-q^i),
$$
where ${\cdot \brack \cdot}_q$ denotes a $q$-binomial coefficient.

\end{theorem}

The classical result of Philip Hall \cite[Thm. 1.10.4]{MR2868112} on conjugacy class size in $\M_n(\Fq)$ can be recovered from the above theorem by setting $k=n$; in this case we necessarily have $\delta=n$ and it follows that $\mu_q(n,n,1;\I)=\gamma_q(n)/c(\I).$

A polynomial matrix $P \in\M_q(n,k,d)$ is said to be {\it unimodular} if the greatest common divisor of all $k \times k$ minors of $P$ is $1$, in other words, if and only if $\delta_k(P)=1$. This corresponds to the case where all invariant factors of $P$ are equal to 1. For $k<n$, a formula for the number of unimodular matrices in $\M_q(n,k,1)$ was given by Helmke, Jordan, and Lieb \cite[Thm. 1]{MR3471061}; this formula may be recovered from Theorem \ref{thmLAA} by setting $\I=(1,\ldots,1)$.

Theorem \ref{thmLAA} can be used to derive an expression for $\sigma(m,d;\I)$ when $\I=(p_1,\ldots,p_m)$ with $\deg p_1 = d-1$ in terms of centralizers in general linear groups.

\begin{corollary}
\label{corLAA}
Suppose $\I=(p_1,\ldots,p_k)\in \mathbb{I}_q(n,k,d)$ and $\deg p_1 = d-1$. Then 
$$
\mu_q(n,k,d;\I)={k \brack \delta}_q \frac{\gamma_q(\delta)}{c(\widetilde{\I})} \prod_{i=\delta+1}^k (q^n-q^i),
$$
where $\widetilde{\I}=(\tilde{p}_1,\ldots,\tilde{p}_k)$ with $\tilde{p}_i=p_i/p_1$ and $\delta=\deg \widetilde{\I}$. 

\end{corollary}

\begin{proof}
By Lemma \ref{star}, we have $\mu_q(n,k,d;\I)=\mu_q(n,k,1;\widetilde{\I})$. The corollary now follows from Theorem \ref{thmLAA}.
\end{proof}

\begin{corollary}
\label{center}
Suppose $\I=(p_1,\ldots,p_m)$ with $\deg \I=md$ and $\deg p_1=d-1$. Then 
$$
\sigma(m,d;\I)=\frac{c(\I)}{c(\widetilde{\I})},
$$
where $\widetilde{\I}=(\tilde{p}_1,\ldots,\tilde{p}_m)$ and $\tilde{p}_i=p_i/p_1$ for $1 \leq i \leq m$. 
\end{corollary}

\begin{proof}
  Note that $\deg \widetilde{\I} = \deg \I -m\cdot \deg p_1 = m$. By Theorem \ref{thm1}, we have
  \begin{align*}
    \sigma(m,d;\I)&=\frac{c(\I)}{\gamma_q(m)}\mu_q(m,m,d;\I)\\
    &=\frac{c(\I)}{\gamma_q(m)}\frac{\gamma_q(m)}{c(\widetilde{\I})},
  \end{align*}
where the last step follows from Corollary \ref{corLAA} by setting $n=m$ and $k=m$.
\end{proof}

By considering types we obtain the following generalization of the above corollary. Recall the definition of $c(\tau)$ from Remark \ref{rem:ctau}. 

\begin{corollary}
\label{ztau}
Let $\tau=\{(d_1,\lambda_1),\ldots,(d_t,\lambda_t)\}$ be a similarity class type of size $md$. If $\sum_{i=1}^t d_i \lambda_{i,m}=d-1$, then
$$
\sigma(m,d;\tau)=\frac{c(\tau)}{c(\widetilde{\tau})},
$$
where $\widetilde{\tau}=\{(d_1,\mu_1),\ldots,(d_t,\mu_t)\}$ and $\mu_i$ is given by $\mu_{i,j}=\lambda_{i,j}-\lambda_{i,m}$ for $1\leq i\leq t$ and $1\leq j\leq m$.
\end{corollary}

Corollary \ref{ztau} may be reformulated in a more explicit form as follows.
\begin{corollary}

Let $\tau=\{(d_1,\lambda_1),\ldots,(d_t,\lambda_t)\}$ be a similarity class type of size $md$. Suppose $\sum_{i=1}^t d_i \lambda_{i,m}=d-1$ and let $m_i$ denote the multiplicity of $\lambda_{i,m}$ as a part of $\lambda_i$ for $1\leq i\leq t$. Then
  \begin{align*}
\sigma(m,d;\tau)&=q^{m^2(d-1)}\prod_{i=1}^t \prod_{j=1}^{m_i}\left(1-{q^{-jd_i}}\right).
\end{align*}
\end{corollary}

\begin{proof}
Continuing with the notation of Corollary \ref{ztau}, we have
\begin{align*}
\sigma(m,d;\tau)=\frac{c(\tau)}{c(\widetilde{\tau})}&= \prod_{i=1}^{t} \frac{c_{d_i}(\lambda_{i})}{c_{d_i}(\mu_{i})} \\
&=\prod_{i=1}^{t} \frac{q^{d_i \langle \lambda'_{i},\lambda'_{i} \rangle}\prod_{j \geq 1}{\left(q^{-d_i}\right)}_{m_j(\lambda_{i})}}{q^{d_i \langle \mu'_{i},\mu'_{i} \rangle}\prod_{j \geq 1}{\left(q^{-d_i}\right)}_{m_j(\mu_{i})}}.
\end{align*}
\begin{figure}[h!]
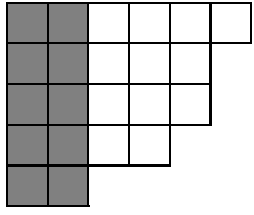

\centering
\ytableausetup{nosmalltableaux}
\begin{ytableau}
  *(gray) & *(gray)  & &  & & \\
  *(gray) & *(gray)  & &  & \\
*(gray) & *(gray)  & & &   \\
*(gray) & *(gray)  & & \\
*(gray) & *(gray) 
\end{ytableau}
\caption{If $m=5$ and $\lambda_{i}=(6,5,5,4,2)$, then $\mu_{i}=(4,3,3,2)$.} 
\label{figure}
\end{figure}
Observe that $\lambda_i'$ may be obtained from $\mu_i'$ by adding $\lambda_{i,m}$ new parts, each equal to $m$. Therefore, $\langle \lambda'_{i},\lambda'_{i} \rangle = m^2 \lambda_{i,m} + \langle \mu'_{i},\mu'_{i} \rangle$. If we remove all parts equal to $\lambda_{i,m}$ from $\lambda_i$, then the multiplicities of the remaining parts coincide with the multiplicities of the parts of $\mu_i$. These observations allow us to rewrite the product above as
\begin{align*}
  \sigma(m,d;\tau)&=\mathlarger\prod_{i=1}^{t}q^{d_i m^2\lambda_{i,m} } \left({q^{-d_i}}\right)_{m_i}\\
                   &=q^{m^2(d-1)}\prod_{i=1}^t \left({q^{-d_i}}\right)_{m_i}.\qedhere
\end{align*}
\end{proof}

We conclude by considering $\mu_q(n,k,d;\I)$ for $k=1$. In this case the problem is equivalent to counting $n$-tuples of coprime monic polynomials of a given degree over a finite field, a question that appears as an exercise in Knuth \cite[Exer. 5 of \S 4.6.1]{MR0286318}. An answer was given by Corteel, Savage, Wilf and Zeilberger \cite[Prop. 3]{MR1620873} (also see \cite[Thm. 4.1]{MR2745427}). 
\begin{proposition}
\label{unimod}
The number of coprime $n$-tuples of monic polynomials of degree $d$ over $\Fq$ is $q^{nd}-q^{n(d-1)+1}$. Equivalently, if $n$ monic polynomials of degree $d$ over $\Fq$ are chosen independently and uniformly at random, then the probability that they are coprime is $1-1/q^{n-1}$.
\end{proposition}

\begin{corollary}
Let $g \in \Fq[x]$ be a monic polynomial of degree $\delta\leq d$. Then
$$
\mu_q(n,1,d;(g))=
\begin{dcases}
  q^{n(d-\delta)}\left(1-q^{1-n}\right) &\delta<d, \\
  1     &\delta =d.
\end{dcases}
$$
\end{corollary}

\begin{proof}
  By Lemma \ref{star}, we have $\mu_q(n,1,d;(g))=\mu_q(n,1,d-\delta;(1))$. If $\delta=d$, then by definition we have $\mu_q(n,1,d-\delta;(1))=1$. If $\delta<d$, then set $d'=d-\delta$ and note that $\mu_q(n,1,d';(1))$ equals the number of polynomial matrices
  \begin{equation}
    \label{eq:k=1}
  \begin{bmatrix}
    g_1\\
    g_2\\
    \vdots\\
    g_n
  \end{bmatrix}
  \end{equation}
such that $\deg g_1 = d'$, $\deg g_i < d'$ for $2 \leq i \leq n$ and $\gcd (g_1,\ldots,g_n)=1$. Since
$$
\gcd (g_1,g_2,\ldots,g_n)=\gcd (g_1,g_2+g_1\ldots,g_n+g_1),
$$
it follows that $\mu_q(n,1,d';(1))$ is equal to the number of coprime $n$-tuples of monic polynomials of degree $d'$. Therefore $\mu_q(n,1,d';(1))=q^{nd'}(1-q^{1-n})$.
\end{proof}

\end{document}